\documentclass[11pt,a4paper]{amsart}
\usepackage[T1]{fontenc}
\usepackage[utf8]{inputenc}
\usepackage{lmodern}
\usepackage{amsmath,amssymb,amsthm,mathtools}
\usepackage{enumitem}
\usepackage{microtype}
\usepackage[a4paper,textwidth=155mm,textheight=235mm,centering]{geometry}
\usepackage[hidelinks,unicode]{hyperref}
\hypersetup{
  pdftitle={Cone Avoidance and the No-Least-Join Theorem},
  pdfauthor={Patrizio Cintioli},
  pdfsubject={Turing degrees, cone avoidance, restricted difference filters},
  pdfkeywords={Turing degrees, computably enumerable degrees, cone avoidance,
  cupping, restricted difference filters, extensions of embeddings}
}
\newtheorem{theorem}{Theorem}[section]
\newtheorem{lemma}[theorem]{Lemma}
\newtheorem{proposition}[theorem]{Proposition}
\newtheorem{corollary}[theorem]{Corollary}
\theoremstyle{definition}
\newtheorem{definition}[theorem]{Definition}
\theoremstyle{remark}

\newcommand{\leT}{\leq_T}
\newcommand{\nleT}{\nleq_T}
\newcommand{\ltT}{<_T}
\newcommand{\eqT}{\equiv_T}
\newcommand{\restr}{\upharpoonright}
\newcommand{\pre}{\preccurlyeq}
\newcommand{\pr}{\prec}
\providecommand{\npreccurlyeq}{\not\preccurlyeq}
\newcommand{\suc}{\succcurlyeq}
\newcommand{\inc}{\perp}
\newcommand{\eps}{\langle\rangle}
\newcommand{\flip}[1]{\widehat{#1}}
\newcommand{\F}{\mathcal F}
\newcommand{\dom}{\operatorname{dom}}
\newcommand{\parr}{\operatorname{par}}
\newcommand{\cat}{\mathbin{{}^{\frown}}}
\newcommand{\DGLM}{\textup{DGLM}}
\newcommand{\zero}{\mathbf 0}
\newcommand{\xx}[3]{\xi_{#1,#2,#3}}
\newcommand{\xs}[3]{\xi^{*}_{#1,#2,#3}}
\setlist[enumerate]{label=\textup{(\roman*)},leftmargin=*,itemsep=3pt,topsep=5pt}
\title[Cone Avoidance and the No-Least-Join Theorem]
{Cone Avoidance and the No-Least-Join Theorem}
\author{Patrizio Cintioli}
\address{Mathematics Division, School of Science and Technology,
University of Camerino, 62032 Camerino, Italy}
\email{patrizio.cintioli@unicam.it}


\begin{document}

\begin{abstract}

We give a column-avoidance modification of the no-least-join construction of
Downey, Greenberg, Lewis, and Montalb\'an. We show that if $C$ is computably
enumerable, $A,B\ltT C$, and $A\nleT B$, then for every noncomputable set $E$
there is a set $X\leT C$ such that $B\leT A\oplus X$ and $E\nleT X$.
Consequently, for the corresponding Turing degrees $a,b,c$, the degrees
$x\leq c$ satisfying $b\leq a\vee x$ have no nonzero common lower bound;
equivalently, no nonzero Turing cone contains all of them.
This gives a negative answer to their Question~1.19.
\end{abstract}
\maketitle

\section{Introduction}

Downey, Greenberg, Lewis, and Montalb\'an~\cite{DGLM} study
extensions of embeddings of finite upper semilattices below computably
enumerable Turing degrees, as part of their analysis of the
two-quantifier theory of the $\Delta^0_2$ degrees with join. One of
the objects isolated in their work is the following upward-closed
family. For degrees $a,b\leq c$, let
\begin{equation}\label{eq:filter}
 \F^c_{a,b}=\{x\leq c:b\leq a\vee x\}.
\end{equation}
Following their terminology, we call this the \emph{restricted
difference filter}. The word ``filter'' refers only to upward closure
inside the initial segment below $c$.

The behavior of these filters already exhibits a sharp contrast.
DGLM's Theorem~1.11 gives c.e. degrees
\[
 \zero<a<b<c
\]
for which $\F^c_{a,b}$ is exactly the cone above $b$, restricted
below $c$. On the other hand, their no-least-join theorem,
Theorems~1.23 and~6.1 of~\cite{DGLM}, shows that if $c$ is c.e.,
$a,b<c$, $a\nleq b$, and $b>\zero$, then
\[
 \F^c_{a,b}\nsubseteq\{x:b\leq x\}.
\]
Thus, under these hypotheses, the degree $b$ itself is not information
that every witness to the cupping requirement must compute.

A natural strengthening is to ask whether \emph{any} fixed nonzero
degree can be unavoidable. Equivalently, can a restricted difference
filter $\F^c_{a,b}$ with $a\nleq b$ be contained in some nonzero
Turing cone?  This question also arises from a second phenomenon
discovered in~\cite{DGLM}. Their Theorem~1.18 constructs pairwise
incomparable c.e. degrees $a,b,d,e$ below a c.e. degree $c$ such that
\[
 \F^c_{a,b}\subseteq\F^c_{d,e}.
\]
Hence one cupping requirement can force another. DGLM ask whether
the information shared by all witnesses to the first requirement
might itself compute a fixed nonzero degree. They formulate the
question as follows.

\begin{quote}
\textbf{Question 1.19 \cite{DGLM}.}
Is there a c.e. set $C$, incomparable sets $A,B<_T C$ and a
noncomputable set $E$ such that for any $\Delta^0_2$ set $X$ which
is computable in $C$, if $X$ joins $A$ above $B$, then $X$ computes
$E$?
\end{quote}

In Turing-reducibility notation, ``$X$ joins $A$ above $B$'' means
$B\leT A\oplus X$, so Question~1.19 asks whether one can find a c.e.
set $C$, incomparable sets $A,B<_T C$, and a noncomputable set $E$
such that
\begin{equation}\label{eq:question}
 (\forall X\leT C)\,
 [B\leT A\oplus X\ \Longrightarrow\ E\leT X].
\end{equation}
The additional requirement that $X$ be $\Delta^0_2$ is automatic
from $X\leT C$, since $C$ is c.e. At the level of degrees,
Question~1.19 therefore asks exactly whether, in the incomparable
case, a restricted difference filter can have a nonzero common lower
bound.

We show that this cannot happen. In fact, no nonzero Turing cone
contains the whole restricted difference filter: for every prescribed
nonzero degree there is a witness to the cupping requirement which
does not compute it.

\begin{theorem}[Cone avoidance]\label{thm:main}
Let $C$ be c.e., let $A,B\ltT C$, and suppose that $A\nleT B$.
For every noncomputable set $E$ there is a set $X\leT C$ such that
\[
 B\leT A\oplus X\qquad\text{and}\qquad E\nleT X.
\]
\end{theorem}

\begin{corollary}\label{cor:filter}
Let $c$ be a c.e. degree, and let $a,b<c$ with $a\nleq b$.
Then $\zero$ is the only common lower bound of $\F^c_{a,b}$.
In particular, Question~1.19 of~\cite{DGLM} has a negative answer.
\end{corollary}
\begin{proof}
For every $e>\zero$, Theorem~\ref{thm:main}, applied to representatives,
gives an $x\in\F^c_{a,b}$ with $e\nleq x$. Thus every common lower bound is $\zero$, and $\zero$ itself is a
lower bound.
The incomparability in Question~1.19 implies the hypothesis $a\nleq b$.
\end{proof}

Equivalently, whenever $a\nleq b$, the restricted difference filter
$\F^c_{a,b}$ is contained in no nonzero Turing cone.
Thus the no-least-join theorem, which excludes $b$ itself as
unavoidable information, extends to cone avoidance for every
prescribed nonzero degree.

This also clarifies the dependence phenomenon of
Theorem~1.18 of~\cite{DGLM}. Although membership in one restricted
difference filter may force membership in another, this implication
cannot, under the hypotheses of Question~1.19, be explained by a fixed
nonzero degree lying below every witness to the first cupping
requirement. In this precise sense, the dependence exhibited by
Theorem~1.18 is genuinely relational.

The order of quantifiers is important. The witness $x$ can depend on $e$.
We do not construct a fixed pair of members of $\F^c_{a,b}$ having no
nonzero common lower bound, and we do not assert a uniform choice of a
successful witness.

The proof is an adaptation of the two-output construction in
\cite[Section~6]{DGLM}. If $E\nleT B$, then $X=B$ already works.
In the remaining case, put $D=B\oplus E\eqT B$. We continue to code the
\emph{whole} of $D$ into each $A\oplus X_j$, but use only the odd column
of $D$ for diagonalization. Merely replacing $B$ by $E$ in the original
no-least-join theorem would not suffice: it would yield a reduction of
$E$, rather than the required reduction of $B$, from $A\oplus X$.

We therefore isolate the part of the construction that does not depend
on its diagonalization target. The structural transfer theorem below
retains the scheduling, coding, and safety mechanism of
\cite[Section~6]{DGLM}, while allowing any finite effective rule for
certificates that obeys the stated structural conditions. Its proof is
included, rather than treating the transfer as a consequence of the
statement of the original theorem.

Question~1.19 arises in DGLM's broader study of the
extension-of-embeddings problem. The present result does not settle
the sufficiency of the anti-cupping condition, nor does it decide the
two-quantifier theory of the $\Delta^0_2$ degrees with join. It does,
however, rule out one natural form that an additional obstruction
might take: a fixed nonzero computational core shared by all witnesses
to a single cupping requirement.

\section{Approximations, strings, and historical records}\label{sec:prelim}

This section fixes the approximation and bookkeeping conventions used
throughout the construction. In particular, we isolate a monotone
approximation to a common modulus, and we introduce the freshness and
frozen-budget mechanisms that will control historical axioms.

\paragraph{Conventions.}
All sets are subsets of $\omega$ and are identified with their
characteristic functions. We write $\leT$ and $\eqT$ for Turing
reducibility and Turing equivalence, $\oplus$ for the usual join, and
$\zero$ for the computable degree. When degree notation is in use,
lower-case letters such as $a,b,c,e$ denote Turing degrees. Standard
background and terminology in computability theory may be found
in~\cite{Soare16}.

Throughout the construction, $A$ and $D$ have computable approximations
$(A_s)$ and $(D_s)$. When $A,D\leT C$ for a c.e. set $C$, these
approximations may be chosen to have a common $C$-computable convergence
bound. Indeed, fix reductions of $A$ and $D$ to $C$, and approximate them
using a computable enumeration $(C_s)$. With oracle $C$, the finitely many
true computations on arguments at most $n$ and their uses can be found.
Waiting until $C_s$ agrees with $C$ below these uses, and until these
computations have run, gives a bound after which both approximations are
correct through $n$. Use a fixed default bit for any computation not yet
observed to converge.

For $s,n\in\omega$, define
\begin{equation}\label{eq:modulus}
 m_s(n)=\min\{t\leq s:(\forall u\in[t,s])\,
       (A_u,D_u)\restr(n+1)=(A_s,D_s)\restr(n+1)\}.
\end{equation}
Thus $m_s(n)$ records the beginning of the current uninterrupted interval
of stability, not the most recent stage belonging to that interval.

\begin{lemma}\label{lem:modulus}
The approximation $(m_s)$ is computable and converges pointwise to a
function $m$. Each coordinate $m_s(n)$ is nondecreasing in $s$.
Moreover,
\begin{equation}\label{eq:change}
 m_{s+1}\restr n\neq m_s\restr n
 \ \Longleftrightarrow\
 (A_{s+1},D_{s+1})\restr n\neq(A_s,D_s)\restr n,
\end{equation}
and, for every finite string $\mu$ of natural numbers,
\begin{equation}\label{eq:nonreturn}
 \mu\pre m_t,\quad \mu\npreccurlyeq m_s,\quad t<s
 \ \Longrightarrow\
 (\forall v\geq s)\ \mu\npreccurlyeq m_v.
\end{equation}
If the two set approximations have a common $C$-computable convergence
bound, then $m\leT C$.
\end{lemma}
\begin{proof}
A change in $(A_s,D_s)\restr(n+1)$ resets $m_s(n)$ to the current
stage; in the absence of a change, its value remains fixed. This proves
monotonicity and convergence. A change below $n$ changes at least the
coordinate $n-1$ of $m_s$, whereas no change below $n$ leaves every
coordinate below $n$ fixed. This proves~\eqref{eq:change}, with the case
$n=0$ immediate. Monotonicity proves~\eqref{eq:nonreturn}.
Finally, given a convergence bound for the first $n+1$ bits of both sets,
compute the finite history up to that bound. The start of its final
stable interval is $m(n)$.
\end{proof}

\paragraph{String notation.}
For binary strings, $\sigma\pre\tau$ means that $\sigma$ is an initial
segment of $\tau$, and $\sigma\pr\tau$ means proper initial segment.
The same notation applies to strings of natural numbers and infinite
sequences. We write $\sigma\inc\tau$ for incompatibility. For a nonempty
binary string $\theta$, $\flip\theta$ is obtained by changing its last
bit. All string coordinates start at $0$.

We enumerate axiom systems $\Xi_0,\Xi_1$ and $\Gamma_0,\Gamma_1$.
An axiom $\mu\mapsto\eta$ in $\Xi_j$ says that an oracle extending $\mu$
produces an output extending $\eta$. An axiom
$(\sigma,\eta)\mapsto\rho$ in $\Gamma_j$ has the analogous meaning for
two binary oracles. The systems $\Gamma_j$ will be consistent on all
oracles. The systems $\Xi_j$ need only be consistent on the approximating
oracles $m_s$ at the relevant stages and on their limit $m$.
Accordingly, before this consistency has been proved, an evaluation of an
axiom system means the collection of its applicable output strings.
Once these strings have been shown to be comparable, their union is the
evaluation. This convention does not presuppose off-path consistency
of $\Xi_j$.

Subscripts $s$ mean that axioms enumerated at stage $s$ are included.
For $j<2$ and $n<\omega$, write
\[
 \xx j s n=\Xi_{j,s}(m_s\restr n),\qquad
 \xs j s n=\Xi_{j,s-1}(m_s\restr n).
\]
We also write $\xi_{j,s}=\Xi_{j,s}(m_s)$ and
$\xi^*_{j,s}=\Xi_{j,s-1}(m_s)$.
Whenever the previous stage evaluations are strings, nonreturn gives
\begin{equation}\label{eq:oldoutputs}
 \xi^*_{j,s}\pre\xi_{j,s-1},\qquad
 \xs j s n\pre\xx j {s-1} n.
\end{equation}
Indeed, an old modulus string that applies at $s$ must still have applied at
$s-1$; once such a string ceases to apply, it cannot return.

\subsection{Fresh lengths and frozen budgets}\label{subsec:budgets}
We keep a computable increasing historical bound $H$. It exceeds the
stage number and all integers and string lengths recorded so far, and is
increased after each finite action. A newly requested \emph{large}
integer exceeds the current bound.
Every nonempty candidate or certificate string ends in $1$, and has
length greater than the bound preceding the action in which it is chosen.
The bound is then increased to cover all the data of that action.

A \emph{free extension} of a string $\eta$ is a string
$\eta\cat0^L1$, where $L$ exceeds the historical bound at the start of
its selection. In particular its length exceeds that of every earlier
candidate or certificate string.

\begin{lemma}[Free extensions]\label{lem:free}
If $\eta'$ is a free extension of $\eta$ and $\zeta$ is a candidate or
certificate string recorded at an earlier stage, then
$\zeta\pre\eta'$ implies $\zeta\pre\eta$.
\end{lemma}
\begin{proof}
Otherwise $\eta\pr\zeta\pre\eta'$. The last bit of $\zeta$ is $1$,
but, by the length bound, its last position lies inside the block of
zeros immediately following $\eta$ in $\eta'$. This is impossible.
The empty string causes no exception.
\end{proof}

One length choice must \emph{not} be made fresh on every repeated call.
For each pair declared correct below, record a \emph{period budget} at
the beginning of each uninterrupted period during which it is correct.
The budget is greater than its node length and all data recorded before
that moment, and is kept fixed throughout that period. At the birth of
a new correct node, its budget is chosen after its defining data have
been recorded. A break in correctness starts a new budget when correctness
resumes. There are only finitely many relevant pairs at each stage, so
these records are effective and finite.

This is an explicit implementation of the frozen choice based on the
beginning of a correct period in \cite[Section~6.3, Step~3]{DGLM}.
Only two properties of these budgets are used: a budget chosen after
some historical record exceeds the size of that record; and, during an
uninterrupted correct period, repeated ordinary calls use the same
budget.

\section{The structural protocol}\label{sec:protocol}

We now describe the abstract two-output protocol. The scheduling,
coding, and safety mechanism is fixed here independently of the
eventual diagonalization target; the only part left open is the
certificate rule introduced below.

A procedure $p$ has a requirement label $r_p$, a modulus string $\mu_p$,
a root $\rho_p$ with $|\rho_p|=|\mu_p|$, two candidate strings
$\eta_{p,0},\eta_{p,1}$, and a partial computable map $\alpha_p$.
Its domain is a binary tree above $\rho_p$: every proper extension of
$\rho_p$ in the domain has its immediate predecessor in the domain.
We will maintain
\[
 |\alpha_p(\beta)|=|\beta|\qquad(\beta\in\dom\alpha_p)
\]
and preservation of the initial-segment relation. Put
$a_p=\alpha_p(\rho_p)$.

The active procedures form a finite chain, starting with $p_0$.
Only the weakest active procedure may call a new procedure. If $q$ is
called by $p$, a node $\beta\in\dom\alpha_p$ is recorded with the call;
we write $\parr(q)=(p,\beta)$ and call $\beta$ the \emph{father} of $q$.
Cancellation always cancels a procedure and all its descendants.
A cancelled procedure is never reactivated, and its historical data
and enumerated axioms are retained.

A pair $(p,\beta)$ is \emph{correct at stage $s$} if $p$ is active,
$\beta\in\dom\alpha_p$, $\beta\pre D_s$, and
$\alpha_p(\beta)\pre A_s$. The pair order is the transitive closure of
\[
 (p,\beta)<(p,\gamma)\quad(\beta\pr\gamma),
 \qquad
 \parr(q)<(q,\rho_q).
\]
We do not assume in the definition that the correct pairs are linearly
ordered; this will follow from the father property.

Every nonroot node $\delta$ added to $\dom\alpha_p$ is supplied with
a \emph{certificate}. It consists of an integer $k_p(\delta)>|\delta|$,
a side $j_p(\delta)<2$, and strings
\[
 \zeta_{p,\delta}(\gamma)
 \quad\text{for all }\gamma\in2^{k_p(\delta)}\text{ with }\delta\pre\gamma.
\]
The explicit subscript $\delta$ prevents any ambiguity between
certificates attached to different nodes.

\begin{definition}[Admissible certificate rule]\label{def:admissible}
Consider a surviving procedure $p$ with current child $q$, and let
$\beta$ be the longest correct node of $p$ before any action at stage $s$.
Put
\[
 \delta=D_s\restr(|\beta|+1),\qquad
 \theta=A_s\restr(|\beta|+1).
\]
When $\delta\notin\dom\alpha_p$, a \emph{structurally admissible
certificate rule} may, by a finite effective test, propose a side $j<2$,
a large $k>|\delta|$, and long strings $\zeta(\gamma)$, ending in $1$,
for all $\delta\pre\gamma\in2^k$, subject to
\begin{equation}\label{eq:safety}
 \nexists(\sigma,\zeta)\in\dom\Gamma_{j,s-1}\,
 [\,\flip\theta\pre\sigma\ \&\ \zeta\pre\eta_{q,j}\,]
\end{equation}
and
\begin{equation}\label{eq:certext}
 \eta_{q,j}\pr\zeta(\gamma)
 \qquad(\delta\pre\gamma\in2^k).
\end{equation}
The rule is allowed to return no proposal. It must terminate at each
stage. No diagonalization property is part of structural admissibility.
\end{definition}

Here and below, comparisons with $A_s,D_s$ are effective, because only
finitely many bits are involved. If the existence of a required longest
node or of a string evaluation has not yet been established, the
instructions are understood subject to the inductive verification in
Section~\ref{sec:transfer}. Safety itself is tested from a finite set
of old axioms: if neither side is safe, there is simply no proposal.

\subsection{Stage instructions}
At stage $0$ create $p_0$, with label $0$, all its strings empty, and
$\alpha_{p_0}(\eps)=\eps$. Enumerate the empty axioms into each
$\Xi_j$ and $\Gamma_j$, and record the first budget for its root.
For $s>0$, perform the following steps.

\paragraph{Step 1: cancellation.}
Cancel every active $p$ for which $\mu_p\npreccurlyeq m_s$, together
with its descendants. Update correctness records and period budgets
for the surviving pairs.

\paragraph{Step 2: certification.}
For every surviving procedure $p$ which has a child, form $\beta,\delta,
\theta$ as in Definition~\ref{def:admissible}. If $\delta$ is not already
in the domain and the rule proposes an admissible certificate, say that
$p$ \emph{requires attention}. If any procedure requires attention, let
$p$ be the strongest such procedure. Record its certificate, set
$\alpha_p(\delta)=\theta$, and record a period budget for this new
correct node. Cancel all procedures weaker than $p$ and end the stage.
No $\Xi$- or $\Gamma$-axioms are enumerated in this case.

\paragraph{Step 3: a new call.}
If no procedure requires attention, let $p$ be the weakest surviving
procedure. Let $\beta$ be its longest correct node, and put
$\delta=D_s\restr(|\beta|+1)$. If $\delta\notin\dom\alpha_p$, let
$k$ be the current frozen period budget of $(p,\beta)$. Otherwise let
$k=k_p(\delta)$. Call a new child $q$ with
\begin{equation}\label{eq:call}
 \begin{gathered}
 r_q=r_p+1,\qquad \parr(q)=(p,\beta),\qquad
 \mu_q=m_s\restr k,\\
 \rho_q=D_s\restr k,\qquad
 a_q=\alpha_q(\rho_q)=A_s\restr k.
 \end{gathered}
\end{equation}
In an \emph{ordinary call}, namely when $\delta$ is undefined, let
$\eta_{q,j}$ be a free extension of $\xs j s k$ for both $j<2$.
In a \emph{certified call}, let $j_*=j_p(\delta)$, let
$\eta_{q,j_*}$ be a free extension of
$\zeta_{p,\delta}(\rho_q)$, and let $\eta_{q,1-j_*}$ be a free
extension of $\xs {1-j_*} s {|\beta|}$.
For both $j<2$, enumerate
\begin{equation}\label{eq:axioms}
 \mu_q\mapsto\eta_{q,j}\in\Xi_j,
 \qquad
 (a_q,\eta_{q,j})\mapsto\rho_q\in\Gamma_j.
\end{equation}
Finally record the root's period budget and complete the historical
bound for the stage. Budgets for existing uninterrupted periods are
not changed.

\section{Structural transfer}\label{sec:transfer}

The purpose of this section is to verify that the preceding protocol
works for every structurally admissible certificate rule. The result
below separates the reusable structural part of the construction from
the particular diagonalization carried out in Section~\ref{sec:column}.

\begin{theorem}[Structural transfer]\label{thm:transfer}
For every structurally admissible finite effective certificate rule,
the protocol in Section~\ref{sec:protocol} has the following properties.
\begin{enumerate}
\item It is well defined at every stage. Each $\Gamma_j$ is consistent,
each $\xi_{j,s}$ is a finite string, and for every active $p$,
\begin{equation}\label{eq:activeidentity}
 \xx j s {|\rho_p|}=\eta_{p,j}\qquad(j<2).
\end{equation}
\item For every $s$ and $\ell>0$, some $j<2$ satisfies
\begin{equation}\label{eq:available}
 \nexists(\sigma,\zeta)\in\dom\Gamma_{j,s}\,
 [\,\flip{(A_s\restr\ell)}\pre\sigma\ \&\ \zeta\pre\xi_{j,s}\,].
\end{equation}
\item If $A\nleT D$, the permanent procedures form an infinite chain,
and the sets $X_j=\Xi_j(m)$ are total and satisfy
\[
 X_j\leT m\qquad\text{and}\qquad \Gamma_j(A,X_j)=D\quad(j<2).
\]
\end{enumerate}
\end{theorem}

We prove the theorem in several steps. These steps isolate the
structural verification of \cite[Section~6]{DGLM}; none uses a property
of a diagonalization functional.

\subsection{Active records and applicable outputs}

We first prove well-definedness and~\eqref{eq:activeidentity} by a
simultaneous induction on the stage. In this induction we also prove
the elementary ancestry facts and the historical extension facts below.
At stage $s$, all assertions at earlier stages are available.
By~\eqref{eq:oldoutputs}, the old evaluations on $m_s$ are already
strings before Step~3 is performed. The verification of the ancestry
facts does not use the new candidate strings chosen in Step~3.

\begin{lemma}[Ancestry and correctness]\label{lem:ancestry}
At every stage the root procedure survives. If $\parr(q)=(p,\beta)$,
then
\[
 |\beta|<|\rho_q|,\quad
 \beta\pr\rho_q,\quad
 \alpha_p(\beta)\pr a_q,\quad \mu_p\pr\mu_q.
\]
The maps $\alpha_p$ preserve proper extension on their domains.
Every active root is correct. More generally, if a historical procedure
$q$ satisfies $\mu_q\pre m_s$, then $\rho_q\pre D_s$ and $a_q\pre A_s$.
While $q$ is active, its father is the longest correct node of its
parent. Consequently, the correct pairs form a linearly ordered initial
segment of the pair order on active procedures.
\end{lemma}
\begin{proof}
The root's modulus string is empty. A new child has root length greater
than its father's length: in an ordinary call this follows from the
budget rule; in a certified call it follows from
$k_p(\delta)>|\delta|>|\beta|$. Its root and its father's node are
prefixes of the same $D_s$, and their labels are prefixes of the same
$A_s$. A new label on a one-bit successor is likewise chosen to extend
the predecessor's label. These facts give the required prefix relations
and, along ancestry, the proper extension of modulus strings.

A modulus string that still applies at stage $s$ has applied at every
stage since its creation, by nonreturn. Equation~\eqref{eq:change}
then shows that the corresponding $A$- and $D$-prefixes have not changed.
This proves the assertion about historical roots and labels, and in
particular about active roots. Correctness is downward closed in the
pair order because both nodes and labels extend along that order.

At its child's creation the father is longest correct by definition.
Suppose it first fails to be longest while the child survives. If a
new correct successor of the father is added, Step~2 cancels the child.
Otherwise an already defined successor has just become correct. The
approximations have then changed below its length, which is at most the
child's root length. Step~1 cancels the child, again a contradiction.
The same argument applies before an action at the next stage when the
child survives Step~1.

Within each individual tree, correct nodes are a downward-closed chain
of prefixes of $D_s$. For two different active procedures, the father
property says that every correct node of the ancestor is at or below
the father leading to the descendant. Iterating along the active chain
therefore compares the two pairs. This proves global linearity, as
well as the asserted initial-segment property.
\end{proof}

If an old procedure $p$ survives at stage $s$, the previous-stage
identity gives
\begin{equation}\label{eq:survivor}
 \xs j s {|\rho_p|}=\eta_{p,j}.
\end{equation}
Indeed $m_s\restr|\rho_p|=m_{s-1}\restr|\rho_p|=\mu_p$.
It follows that a new ordinary child extends its parent's candidate
strings. The same holds in a certified call: the certificate extends
the candidate of the child that was active at its registration, and
that old candidate extended the parent's candidate by induction.
Thus, at the current stage,
\begin{equation}\label{eq:activegrowth}
 p<q\text{ active}\quad\Longrightarrow\quad
 \eta_{p,j}\pr\eta_{q,j}\quad(j<2).
\end{equation}

The remaining issue for $\Xi_j$ is the contribution of cancelled
procedures whose modulus strings still apply.

\begin{lemma}[Persistence of an applicable historical modulus string]\label{lem:persistence}
Suppose that $q$ was active at a stage $t<s$, is cancelled by the end
of stage $s$, and $\mu_q\pre m_s$. Put $N=|\rho_q|$.
Then
\begin{equation}\label{eq:persistence}
 \xs j s N=\Xi_{j,t}(\mu_q)=\eta_{q,j}\quad(j<2).
\end{equation}
If a new procedure $p$ is called at stage $s$, then
\begin{equation}\label{eq:historicalgrowth}
 |\rho_p|>N\quad\text{and}\quad\eta_{q,j}\pr\eta_{p,j}\quad(j<2).
\end{equation}
\end{lemma}
\begin{proof}
All the assertions are included in the stage induction. Let $r$ be the
weakest procedure active at both $t$ and $s$. It exists because the
root survives, and it is a proper ancestor of $q$. Let $v$ be the
child of $r$ active at $t$, and let $\beta$ be its father. The procedure
$v$ is cancelled at some stage $u\in(t,s]$. Since
$\mu_v\pre\mu_q\pre m_s$, the modulus string $\mu_v$ remains applicable
throughout this interval.

The procedure $r$ also survives throughout. Thus $v$ is
cancelled by an action of $r$ at Step~2 of stage $u$, not by Step~1
or by an ancestor of $r$.

The father property identifies the new node as
$\beta^+=D_u\restr(|\beta|+1)$. It is correct at $u$ and remains
correct through $s$: its length is at most $N$, and the approximations
below $N$ have not changed since $t$. Any child of $r$ active at $s$
therefore has father extending $\beta^+$. Such a father was added to
$\dom\alpha_r$ at or after stage $u$.

Suppose a procedure $p$ is called at stage $s$. If its parent is $r$,
its father $\delta$ extends $\beta^+$. In an ordinary call, the budget
of its present correct period was chosen no earlier than the addition
of $\delta$, so it exceeds the already recorded length $N$.
In a certified call, the certificate of the successor of $\delta$
was registered still later, so its large length also exceeds $N$.
If $p$ is further below $r$, the current child of $r$ was created after
$u$ and before $s$. By the stage induction its root length exceeds $N$;
so does the root length of $p$. This proves the length assertion
in~\eqref{eq:historicalgrowth}.

For every intervening stage $w\in(t,s)$, if $q$ is still active then a
new call at $w$ is below $q$ and has a modulus string of length greater
than $N$. If $q$ is already cancelled, the same statement follows from
the length assertion at stage $w$ by induction, since $\mu_q\pre m_w$.
Consequently no axiom $\mu\mapsto\eta$ with $|\mu|\leq N$ is enumerated
between stages $t$ and $s-1$. Moreover,
$m_t\restr N=m_s\restr N=\mu_q$. Hence
\[
 \xs j s N=\Xi_{j,t}(\mu_q).
\]
At stage $t$, the active-root identity identifies this last string with
$\eta_{q,j}$, proving~\eqref{eq:persistence}. This is an equality of
historical evaluations, not merely a comparison of candidate strings.

It remains to verify the candidate extension when $p$ is called at $s$.
If $p$ is below a child of $r$ created before $s$, use the extension
assertion at that earlier creation stage and~\eqref{eq:activegrowth}.
We may therefore assume that the parent of $p$ is $r$, with father
$\delta$ added no earlier than $u$.

For an ordinary call, the base is $\xs j s {|\rho_p|}$, which extends
$\xs j s N=\eta_{q,j}$. For a certified call, put
$\delta^+=D_s\restr(|\delta|+1)$ and let $w<s$ be the stage at which
$\delta^+$ was certified. Since $\delta$ was added at or after $u$,
we have $u<w<s$. The node $\delta^+$ was correct at $w$ but is not
correct at $s$, while all prefixes of length $N$ remained fixed.
Therefore $|\delta^+|>N$, so
\begin{equation}\label{eq:fatherN}
 |\delta|\geq N.
\end{equation}
The side using a free extension of $\xs j s {|\delta|}$ thus extends
$\eta_{q,j}$ by~\eqref{eq:persistence}.
For the certificate side, let $o$ be the child of $r$ at stage $w$
before certification. It survives Step~1, and
$|\rho_o|>|\delta|\geq N$. By the historical equality at stage $w$
and the active-root identity at stage $w-1$,
\[
 \eta_{q,j}=\xs j w N
 \pre\xs j w {|\rho_o|}=\eta_{o,j}.
\]
Every certificate string on that side extends $\eta_{o,j}$.
The free extension used for $p$ therefore properly extends
$\eta_{q,j}$ as well. This completes the simultaneous proof.
\end{proof}

\begin{proposition}\label{prop:Xi}
At every stage, $\xi_{j,s}$ is a finite string, and
\eqref{eq:activeidentity} holds for every active procedure.
\end{proposition}
\begin{proof}
If a new procedure $p$ is called, every old modulus string applying
to $m_s$ comes from either a surviving ancestor or a cancelled historical
procedure. Its associated candidate string is properly extended by
$\eta_{p,j}$, respectively by~\eqref{eq:activegrowth} or
Lemma~\ref{lem:persistence}.

Thus the new evaluation is the string $\eta_{p,j}$.
If there is no call, no new axiom is enumerated, and consistency follows
from~\eqref{eq:oldoutputs}.

For an old surviving $q$, a new call has root length greater than
$|\rho_q|$, so it does not change the evaluation on the finite oracle prefix
$m_s\restr|\rho_q|$. Equation~\eqref{eq:survivor} proves the required
identity. For a newly called $p$, its own axiom applies and all old strings
supplied by applicable axioms are shorter prefixes of $\eta_{p,j}$, so the
identity also holds. This completes the stage induction, and therefore
justifies every finite evaluation used by the protocol.
\end{proof}

\subsection{Consistency of the coding functionals}

We now use a second stage induction. Its assertions are the consistency
of $\Gamma_{j,s}$, the comparison property in
Lemma~\ref{lem:Gamma}, and the certificate property in
Lemma~\ref{lem:certaxioms}. The construction is already well defined
by Proposition~\ref{prop:Xi}. At a certification stage we prove the
certificate property from the previous consistency assertion; at a call
stage we use already registered certificates to prove the comparison
property. This order avoids a dependency on future consistency.

The next lemma records exactly what the safety test at a certification
stage guarantees about previously enumerated coding axioms. It is the
local fact used in the certified-call case of the consistency proof
for the $\Gamma_j$.

\begin{lemma}[Old axioms at a certificate]\label{lem:certaxioms}
Suppose that $r$ certifies $\delta=\beta\cat i$ at stage $t$, with
$\alpha_r(\delta)=\theta$ and selected side $j$. Let $v$ be the child
of $r$ just before that action. For each certificate string
$z=\zeta_{r,\delta}(\gamma)$, the following statements hold.
\begin{enumerate}
\item $\parr(v)=(r,\beta)$ and $\theta\pre a_v$.
\item If $(\sigma,\zeta)\mapsto\nu$ is in $\Gamma_{j,t-1}$,
$\zeta\pre z$, and $\theta\npreccurlyeq\sigma$, then
$\zeta\pre\eta_{v,j}$.
\item If $(\sigma,\zeta)\mapsto\nu$ is in $\Gamma_{j,t-1}$,
$\zeta\pre z$, and $\sigma$ is comparable with $\flip\theta$, then
$\nu\pre\beta$.
\end{enumerate}
Furthermore, any coding axiom at a later stage whose second component
is a prefix of $z$ was already present at stage $t-1$.
\end{lemma}
\begin{proof}
The father property gives $\parr(v)=(r,\beta)$. The child survives
Step~1, so $a_v\pre A_t$. Since $|\rho_v|>|\beta|$ and
$\theta=A_t\restr(|\beta|+1)$, we have $\theta\pre a_v$.

For (ii), the strings $\zeta$ and $\eta_{v,j}$ are comparable, since
both are prefixes of $z$. Suppose that $\eta_{v,j}\pr\zeta$.
Write $\sigma=a_q$, $\zeta=\eta_{q,j}$, and $\nu=\rho_q$ for the
procedure responsible for the old axiom. Freshness of candidate lengths
implies that $q$ was called after $v$. Since $v$ remained active up to
the certification, $q$ was a descendant of $v$ when it was called.
The ancestry relations give $a_v\pr a_q$, contrary to
$\theta\npreccurlyeq\sigma$.

For (iii), a string comparable with $\flip\theta$ cannot extend
$\theta$. Part (ii) therefore gives $\zeta\pre\eta_{v,j}$.
The safety condition excludes $\flip\theta\pre\sigma$, so
\[
 \sigma\pr\flip\theta,
 \qquad\text{hence}\qquad \sigma\pre\alpha_r(\beta).
\]
The axiom in question is compatible with the old child axiom
$(a_v,\eta_{v,j})\mapsto\rho_v$. By the inductive consistency of
$\Gamma_{j,t-1}$, its output $\nu$ is compatible with $\rho_v$.
Every coding axiom has equal first-component and output lengths. Thus
\[
 |\nu|=|\sigma|\leq|\beta|,
 \qquad
 \nu\pre\rho_v\restr|\beta|=\beta.
\]
This is the required short-output conclusion.

Finally, no call occurs at stage $t$, and a candidate created after
$t$ is longer than $z$. It cannot be the second component of an axiom
whose second component is a prefix of $z$.
\end{proof}

\begin{lemma}[Comparison of coding axioms]\label{lem:Gamma}
For any historical procedures $p,q$ and any $j<2$,
\begin{equation}\label{eq:Gammacompare}
 \eta_{q,j}\pre\eta_{p,j}
 \quad\Longrightarrow\quad
 a_q\inc a_p\ \text{ or }\ \rho_q\pre\rho_p.
\end{equation}
Consequently $\Gamma_j$ is a consistent Turing functional.
\end{lemma}
\begin{proof}
Consider a new procedure $p$ at stage $s$. For $q\neq p$ with
$\eta_{q,j}\pre\eta_{p,j}$, freshness implies that $q$ was called
before $s$. We first record the following consequence of the inductive
comparison property:
\begin{equation}\label{eq:oldbaselemma}
 \eta_{q,j}\pre\xs j s n
 \quad\Longrightarrow\quad
 a_q\inc(A_s\restr n)\ \text{ or }\ \rho_q\pre D_s\restr n.
\end{equation}
To prove it, choose a historical procedure $w$ whose modulus string applies to
$m_s\restr n$ and whose candidate $\eta_{w,j}$ is the longest applicable
$\Xi_j$-string there. Such a $w$ exists, including the root when this
string is empty.
Then $|\rho_w|\leq n$ and $\eta_{q,j}\pre\eta_{w,j}$.
By the comparison property at earlier stages, either $a_q\inc a_w$
or $\rho_q\pre\rho_w$. Since $\mu_w\pre m_s$, we have
$a_w\pre A_s\restr n$ and $\rho_w\pre D_s\restr n$.
This proves~\eqref{eq:oldbaselemma}.

If $\eta_{p,j}$ is a free extension of $\xs j s n$, where $n$ is
its root length or its father's length, Lemma~\ref{lem:free} reduces
$\eta_{q,j}\pre\eta_{p,j}$ to $\eta_{q,j}\pre\xs j s n$.
The first case of~\eqref{eq:oldbaselemma} gives incompatibility with
$a_p$, and the second gives $\rho_q\pre\rho_p$.

Otherwise $p$ is a certified call on its parent's selected side.
Write its parent as $r$, its father as $\beta$, and
$\delta=D_s\restr(|\beta|+1)$. Let $t<s$ be the certificate stage
and $\theta=\alpha_r(\delta)$. By maximal correctness of the father,
$\delta$ is not correct now. Its predecessor is correct and its node
is a prefix of $D_s$. Therefore
\begin{equation}\label{eq:wronglast}
 \flip\theta\pre a_p.
\end{equation}
By the free-extension lemma, $\eta_{q,j}$ is a prefix of the relevant
certificate string. Its coding axiom was already present before $t$,
by the last assertion of Lemma~\ref{lem:certaxioms}. If $a_q$ is
compatible with $a_p$, it is comparable with $\flip\theta$.
Part (iii) of that lemma yields
$\rho_q\pre\beta\pr\rho_p$. This proves~\eqref{eq:Gammacompare}
for the new call. At stages without a call, there is no new pair of
coding axioms to check. Together with Lemma~\ref{lem:certaxioms}, this
completes the second induction.

Finally, two axioms can apply to a common pair of oracles only if their
first components are compatible and their second components are
comparable. Apply~\eqref{eq:Gammacompare} in the direction of the
shorter second component. Their outputs are comparable, which is
exactly consistency.
\end{proof}

\begin{corollary}[No deeper old axiom]\label{cor:nodeeper}
For every historical $p$ and $j<2$, there is no coding axiom
$(\sigma,\zeta)\mapsto\nu$ in $\Gamma_j$ such that
$a_p\pr\sigma$ and $\zeta\pre\eta_{p,j}$.
\end{corollary}
\begin{proof}
Write the axiom as $(a_q,\eta_{q,j})\mapsto\rho_q$.
Lemma~\ref{lem:Gamma} would give $\rho_q\pre\rho_p$, whereas
$a_p\pr a_q$ and equality of first-component and output lengths give
$|\rho_p|<|\rho_q|$.
\end{proof}

\subsection{The history of a certified call}

A certified call may occur long after its certificate was registered.
The following lemma identifies what can have happened in the intervening
stages and, in particular, shows that no procedure called after the
certificate stage can still have an applicable modulus string at the
time of the call.
\begin{lemma}\label{lem:history}
Suppose $p$ is called at stage $s$ using a certificate registered at
stage $t<s$ by its parent $r$. Write $\parr(p)=(r,\beta)$ and
$\delta=\beta\cat i$ for the certified node. Then:
\begin{enumerate}
\item The first position at which $A_t$ and $A_s$ differ is $|\beta|$.
\item Every old procedure $q$ with $\mu_q\pre m_s$ was called before
stage $t$. Consequently, $\xi^*_{j,s}\pre\xi_{j,t}$ for both sides.
\item If an axiom in $\Gamma_{j,s-1}$ has its second component contained
in $\xi_{j,t}$ or in any certificate string registered at $t$, then
that axiom already belongs to $\Gamma_{j,t}$.
\end{enumerate}
\end{lemma}
\begin{proof}
The label of $\beta$ is a prefix of both $A_t$ and $A_s$.
At stage $t$, the label of $\delta$ extends it correctly by one bit.
At stage $s$, the node $\delta$ is a prefix of $D_s$ but its label is
incorrect, since $\beta$ is longest correct. This proves (i).

Here is a direct proof of (ii) using the persistence mechanism.
Suppose that an old $q$ was called after $t$ and that
$\mu_q\pre m_s$. The procedure $r$ survives from $t$ to $s$ and
cancels all its descendants at $t$. Thus $q$ was called below $r$.
Let $v$ be the immediate child of $r$ on the ancestry of $q$.
It was called after $t$, and $\mu_v\pre\mu_q\pre m_s$.
It is no longer active when $p$ is called. Because its modulus string
has never ceased to apply and $r$ survives, its cancellation must be a Step~2
action of $r$ at some stage $u\in(t,s)$.

Let $\lambda$ be the father of $v$. At stage $u$, $r$ adds the
correct node $\lambda^+$ extending $\lambda$ by one bit.
Since $|\lambda^+|\leq|\rho_v|$, this node and its label remain
correct through $s$. Therefore the longest correct node $\beta$ at
stage $s$ extends $\lambda^+$. But $\beta$ already belonged to
$\dom\alpha_r$ before stage $t$, while $\lambda^+$ was first
added after $t$. This contradicts the tree property of the domain.
No call occurs at $t$ itself. Hence all old applicable procedures
were called before $t$.

Their modulus strings also applied at stage $t$, by nonreturn.
Taking the union of their candidate strings proves
$\xi^*_{j,s}\pre\xi_{j,t}$.
Part (iii) follows from freshness: a candidate chosen after $t$ is
longer than every string mentioned in (iii), so cannot be their prefix.
\end{proof}

\subsection{An available side at every stage}

The central safety invariant is that, for every relevant flipped
$A_s$-prefix, at least one of the two sides remains free of a conflicting
old coding axiom. This is the mechanism that guarantees the existence
of a safe side whenever a procedure is ready to act.

\begin{lemma}[Available-side invariant]\label{lem:available}
For every stage $s$ and every $\ell>0$, some $j<2$ satisfies
\eqref{eq:available}.
\end{lemma}
\begin{proof}
Induct on $s$. The assertion holds at stage $0$. At stage $s$, a new
coding axiom has first component a prefix of $A_s$, so it cannot
extend $\flip{(A_s\restr\ell)}$. It is therefore enough to exclude
old coding axioms.

If there is no call at $s$, or the call is ordinary, put $t=s-1$.
If a certified call is made, let $t$ be its certificate-registration
stage. Let $n$ be the length of the longest common initial segment of
$A_t$ and $A_s$, allowing $n=\infty$ when they are equal.
For a certified call, Lemma~\ref{lem:history}(i) says that $n$ is
its father's length.

First suppose $\ell\leq n$. At stage $t$ choose a side $j$ that
satisfies the invariant for $\ell$. The flipped prefixes at stages
$t$ and $s$ agree. Suppose, toward a contradiction, that an old axiom
$(\sigma,\zeta)\mapsto\nu$ has
\[
 \flip{(A_s\restr\ell)}\pre\sigma,
 \qquad \zeta\pre\xi_{j,s}.
\]
With no call, $\xi_{j,s}=\xi^*_{j,s}\pre\xi_{j,t}$, contradicting
our choice at $t$. With an ordinary call, the free-extension lemma
reduces $\zeta$ to a prefix of the old base, hence of $\xi_{j,t}$;
again we have a contradiction.

For a certified call, first consider the side that uses a free
extension of the old base truncated at the father's length. The
free-extension lemma and Lemma~\ref{lem:history}(ii) reduce $\zeta$
to a prefix of $\xi_{j,t}$; part (iii) of that lemma puts its axiom
at stage $t$. This contradicts the invariant at $t$.
On the selected side, the free-extension lemma instead puts $\zeta$
below a certificate string. Its axiom was already present at $t$.
Let $\theta$ be the label registered with that certificate, and let
$v$ be the child present at its registration. Since $\ell\leq n$
and $\sigma$ extends the flipped prefix at length $\ell$, it cannot
extend $\theta$, which is a prefix of $A_t$ of length $n+1$.
Lemma~\ref{lem:certaxioms}(ii) gives
$\zeta\pre\eta_{v,j}\pre\xi_{j,t}$. The last inclusion holds even
though $v$ was cancelled at $t$: its modulus string still applied at that
certification stage and no new $\Xi$-axiom was enumerated there.
This is the remaining contradiction for $\ell\leq n$.

Now suppose $\ell>n$. Choose a side $j$ on which $\xi_{j,s}$ is
either $\xi^*_{j,s}$ or a free extension of an initial segment of
$\xi^*_{j,s}$. Both sides qualify with no call or an ordinary call;
in a certified call, use the side not selected by the certificate.
Any old second component $\zeta\pre\xi_{j,s}$ is consequently a
prefix of $\xi^*_{j,s}$.

Choose an old applicable procedure $q$ such that
$\eta_{q,j}=\xi^*_{j,s}$, using the root if this string is empty.
In the first two cases $q$ was called by $t=s-1$; in the certified
case it was called before $t$ by Lemma~\ref{lem:history}(ii).
Since its modulus string applies at both $t$ and $s$, its root and label
are prefixes of both approximations. In particular $|\rho_q|\leq n$,
as $A_t$ and $A_s$ disagree at $n$. Thus
\[
 a_q\pre A_s\restr n\pr\flip{(A_s\restr\ell)}.
\]
An old axiom whose first component extends this flipped prefix and
whose second component is contained in $\xi^*_{j,s}=\eta_{q,j}$
would violate Corollary~\ref{cor:nodeeper}. This proves the invariant.
\end{proof}

\begin{corollary}[Safety before action]\label{cor:preaction}
Suppose a child $q$ is active at the end of stage $s-1$ and survives
Step~1 of stage $s$. Let $\parr(q)=(p,\beta)$ and
$\theta=A_s\restr(|\beta|+1)$. At least one side satisfies the safety
condition~\eqref{eq:safety} with respect to the old axioms.
\end{corollary}
\begin{proof}
Put $\ell=|\beta|+1\leq|\rho_q|$. Survival implies
$A_s\restr\ell=A_{s-1}\restr\ell$. Apply
Lemma~\ref{lem:available} at stage $s-1$, not at stage $s$.
For both sides $\eta_{q,j}\pre\xi_{j,s-1}$, so the side it supplies
is safe for the smaller candidate prefix as well.
\end{proof}

\subsection{Permanent procedures and coding at the limit}

Assume now that $A\nleT D$. Call a procedure \emph{permanent} if it
is never cancelled. For such a procedure $p$, a node $\beta$ is
\emph{truly correct} if $\beta\pre D$ and $\alpha_p(\beta)\pre A$.
This is equivalent to the pair being correct at all sufficiently late
stages, and also to its being correct at infinitely many stages.

\begin{lemma}\label{lem:permanent}
Every permanent procedure has a longest truly correct node, requires
attention at only finitely many stages, and has a permanent child.
\end{lemma}
\begin{proof}
The root of $p$ is truly correct, because $\mu_p$ applies at every
later stage. If truly correct nodes were unbounded in length, the tree and
label-extension properties would imply that every defined node
$\nu\pre D$ has its correct $A$-label: choose a longer truly correct
node and use preservation of initial segments. With oracle $D$, given
$n$, enumerate $\alpha_p$ until a node
$\nu\in\dom\alpha_p$ with $\nu\pre D$ and $|\nu|>n$ is found, and
output $\alpha_p(\nu)(n)$. The search terminates by unboundedness and
is correct by the preceding observation. This computes $A$ from $D$,
with the finite index of $p$ and its root built into the reduction.

The hypothesis excludes this. Let $\beta$ be the longest
truly correct node and put $\delta=D\restr(|\beta|+1)$.

Choose $s_0$ after $\beta$ has been defined and after the approximations
below $|\delta|$ have settled. If $\delta$ is ever defined, also choose
$s_0$ after its definition. At all stages beyond $s_0$, the longest
correct node of $p$ is $\beta$. If $\delta$ is defined, Step~2 does
not request another certificate for it. If $\delta$ is never defined,
a late request would lead to its definition: no stronger procedure
can act, since that would cancel the permanent $p$. This would
contradict our choice of $\delta$ as never defined. Thus $p$ no longer
requires attention.

After $s_0$, whenever $p$ has no child following Step~1, it is the
weakest active procedure and calls a child in Step~3. The possible
root lengths of these new children are bounded. They use either the
fixed certificate length for $\delta$, or the fixed period budget
of $\beta$, whose correctness no longer breaks. Choose a later stage
after $m_s$ has settled below this bound. A child created after that
stage cannot be cancelled by Step~1. It cannot be cancelled by $p$,
which no longer requires attention, or by an ancestor, since $p$
survives. Hence it is permanent. If no new child is created after
that stage, the already present child is itself permanent.
\end{proof}

\begin{proof}[Completion of Theorem~\ref{thm:transfer}]
Proposition~\ref{prop:Xi} and Lemmas~\ref{lem:Gamma}
and~\ref{lem:available} establish the first two assertions.
Starting from the permanent root and repeatedly applying
Lemma~\ref{lem:permanent}, choose permanent children
$p_0,p_1,\ldots$, with $r_{p_r}=r$.
On each side $j<2$, their candidate strings $\eta_{p_r,j}$ strictly
increase with $r$, and every corresponding modulus string is a prefix of $m$.

Any two $\Xi_j$-axioms whose left-hand strings are prefixes of $m$ are
both present and applicable at a sufficiently late stage.
Proposition~\ref{prop:Xi}
then makes their outputs comparable. Thus $\Xi_j(m)$ is a consistent
union, and the permanent candidates make it infinite. To compute its
$n$th bit with oracle $m$, enumerate axioms until one is
found whose left-hand string is a prefix of $m$ and whose output has
length greater than $n$. Totality guarantees
termination and the established consistency guarantees correctness.
This proves $X_j=\Xi_j(m)\leT m$, without asserting consistency of
$\Xi_j$ on unrelated oracles.

Finally, $\rho_{p_r}\pre D$, $a_{p_r}\pre A$, and
$\eta_{p_r,j}\pre X_j$. The corresponding coding axioms therefore
apply to $(A,X_j)$. Their output lengths strictly increase, so they
cover $D$. Consistency of $\Gamma_j$ implies
$\Gamma_j(A,X_j)=D$, as required.
\end{proof}

\section{Avoiding a designated column}\label{sec:column}

We now specialize the structural protocol. The key point is that the
coding functionals will continue to recover the whole target $D$,
while the certificate rule diagonalizes only against one designated
column of $D$. This separation is what allows the coded set and the
avoided set to play different roles.

For an infinite binary sequence $D$, define its odd column by
$\pi(D)(n)=D(2n+1)$. For a finite binary string $\gamma$, $\pi(\gamma)$
has the same definition on those $n$ with $2n+1<|\gamma|$.
This is a literal finite projection, including on incorrect
approximations. It is not an application of a possibly divergent
Turing reduction to an arbitrary finite oracle.

\begin{lemma}[Column avoidance]\label{lem:column}
Let $C$ be c.e., let $A,D\leT C$, and suppose that $A\nleT D$.
If $E=\pi(D)$ is noncomputable, there are sets $X_0,X_1\leT C$ such that
\[
 D\leT A\oplus X_j\quad(j<2),
 \qquad
 \neg(E\leT X_0\ \&\ E\leT X_1).
\]
\end{lemma}
\begin{proof}
Choose approximations and the modulus as in Section~\ref{sec:prelim},
so $m\leT C$.
Effectively enumerate all ordered pairs of binary-valued
Turing functionals as $(\Phi_{r,0},\Phi_{r,1})$, for $r\in\omega$.
We meet
\begin{equation}\label{eq:requirements}
 \mathcal R_r:\quad
 \Phi_{r,0}^{X_0}\neq E\quad\text{or}\quad\Phi_{r,1}^{X_1}\neq E.
\end{equation}
The enumeration includes pairs with different indices on the two sides.
If $\Phi_{r,j}^{X_j}$ is partial, then it is of course unequal to the
total characteristic function of $E$.

We specify the certificate rule for a procedure $p$ with label $r$ and
current child $q$. At stage $s$, on each side inspect all witnesses,
with codes and computation times at most $s$, to a splitting above
$\eta_{q,j}$. Such a witness consists of finite strings
$\tau^0,\tau^1\suc\eta_{q,j}$ and an argument $n$ with
\begin{equation}\label{eq:split}
 \Phi_{r,j}^{\tau^0}(n)\downarrow=0,
 \qquad
 \Phi_{r,j}^{\tau^1}(n)\downarrow=1.
\end{equation}
The computations in~\eqref{eq:split} use only positions within the
specified finite strings. The bounded tests are exhaustive as $s$
increases for any fixed candidate base.

If a splitting is found on a safe side $j$, select one by a fixed
effective order. Choose a fresh large $k$ such that
\[
 k>|\delta|\qquad\text{and}\qquad k>2n+1,
\]
where $\delta=D_s\restr(|\beta|+1)$ is the node to be certified.
For every $\delta\pre\gamma\in2^k$, choose a sufficiently long
extension $\zeta(\gamma)$, ending in $1$, of
\[
 \tau^{\,1-\gamma(2n+1)}.
\]
All these extensions can be chosen effectively after the bounded
splitting test. Their lengths need not be at most $s$; the action is
finite, not time-bounded by the stage number. They satisfy
\begin{equation}\label{eq:columncertificate}
 \Phi_{r,j}^{\zeta(\gamma)}(n)=1-\gamma(2n+1).
\end{equation}
They properly extend $\eta_{q,j}$ and obey the freshness convention.
If neither safe side has a splitting observed by that stage, return
no proposal. This is a finite effective structurally admissible rule.
For this application, whenever such a proposal is acted upon at Step~2,
we also retain its splitting argument $n$ with the registered certificate.
This is auxiliary data for the present verification and plays no
role in the structural protocol.

Theorem~\ref{thm:transfer} now gives total sets $X_j\leT m\leT C$
and $\Gamma_j(A,X_j)=D$ for both $j$. Fix $r$, let $p$ be the
permanent procedure with label $r$, and let $\beta$ be its longest
truly correct node. Put $\delta=D\restr(|\beta|+1)$, and let $q$
be the permanent child of $p$. Its father is $\beta$: the father
property holds at all late stages, where the longest correct node
has stabilized.

Suppose first that $\delta\in\dom\alpha_p$. Its certificate must
have been registered before $q$ was called, since a later registration
would cancel $q$. Write its data as $j,k,n$ and
$\zeta_{p,\delta}(\gamma)$. The call of $q$ is a certified call, with
\[
 \rho_q=D\restr k,
 \qquad
 \zeta_{p,\delta}(\rho_q)\pr\eta_{q,j}\pre X_j.
\]
Here we use the \emph{long root} $\rho_q$, not just the one-bit
successor $\delta$. Since $k>2n+1$, equation~\eqref{eq:columncertificate}
gives
\begin{equation}\label{eq:finaldiagonal}
 \Phi_{r,j}^{X_j}(n)
 =1-\rho_q(2n+1)=1-E(n).
\end{equation}
Thus $\mathcal R_r$ is met.

Suppose next that $\delta\notin\dom\alpha_p$, and toward a
contradiction assume
\begin{equation}\label{eq:bothcompute}
 \Phi_{r,0}^{X_0}=E=\Phi_{r,1}^{X_1}.
\end{equation}
There must be a splitting above $\eta_{q,j}$ on each side $j$.
Indeed, if there were no splitting on a side, a computable procedure
with the finite string $\eta_{q,j}$ built into its program could
compute $E(n)$ as follows: search for a finite extension of this
string on which $\Phi_{r,j}(n)$ converges, and return the resulting
bit. Equality~\eqref{eq:bothcompute} supplies such an extension
along $X_j$, so the search terminates. The absence of a splitting
makes its answer unique and equal to the value along $X_j$.
This would make $E$ computable, a contradiction.

Choose a sufficiently late stage at which $p$ and $q$ are present,
the longest correct node is $\beta$, and the bounded searches have
found a splitting on \emph{both} sides above the now fixed bases
$\eta_{q,j}$. Corollary~\ref{cor:preaction} supplies at least one
safe side at this stage. The rule therefore makes $p$ require
attention. No stronger procedure can act without cancelling the
permanent $p$, so $p$ acts and defines $\delta$. This contradicts
the present case. Hence~\eqref{eq:bothcompute} is impossible, and
$\mathcal R_r$ is met.

Every requirement is met. If both $X_0$ and $X_1$ computed $E$, the
ordered pair of functionals witnessing these two reductions would
occur in our enumeration and violate its requirement. Therefore
at least one side fails to compute $E$.
\end{proof}

\section{The cone-avoidance conclusion}\label{sec:conclusion}

We finish by reducing the general cone-avoidance statement to the
column-avoidance lemma. The only nontrivial case is when the set to
be avoided is already computable from $B$; we then place it in a
separate column of a target Turing-equivalent to $B$.

\begin{proof}[Proof of Theorem~\ref{thm:main}]
Fix a noncomputable $E$. If $E\nleT B$, take $X=B$.
Then $X\leT C$, $B\leT A\oplus X$, and $E\nleT X$.

Otherwise $E\leT B$. Define the literal join
\[
 D(2n)=B(n),\qquad D(2n+1)=E(n).
\]
Then $D=B\oplus E\eqT B$, so $A,D\leT C$ and $A\nleT D$.
The odd column of $D$ is the noncomputable set $E$.
Lemma~\ref{lem:column} gives two candidates $X_j\leT C$ with
$D\leT A\oplus X_j$, at least one of which does not compute $E$.
For that side let $X=X_j$. Then
\[
 B\leT D\leT A\oplus X\qquad\text{and}\qquad E\nleT X.
\]
\end{proof}

\section{Concluding remarks}\label{sec:concluding}

Theorem~\ref{thm:main} shows that the no-least-join phenomenon extends
to cone avoidance for every prescribed nonzero degree in the following
precise sense: whenever $c$ is c.e. and $a,b<c$ with $a\nleq b$, the
restricted difference filter
\[
    \F^c_{a,b}=\{x\leq c:b\leq a\vee x\}
\]
is contained in no nonzero Turing cone. Equivalently, $\zero$ is its
only common lower bound. Thus Question~1.19 of~\cite{DGLM} has a
negative answer.

This also sharpens the interpretation of the dependence phenomenon in
Theorem~1.18 of~\cite{DGLM}. One restricted difference filter may be
contained in another even though the first has no nonzero common lower
bound. Hence such a dependence cannot, under the hypotheses considered
here, be explained by a fixed nonzero degree computed by every witness
to the first cupping requirement.

The hypothesis $a\nleq b$ is essential in this form.
Theorem~1.11 of~\cite{DGLM} gives configurations with $a\leq b$ in
which $b$ itself is a nonzero common lower bound. The present result
therefore rules out one natural source of additional extension
obstructions, but does not provide a general characterization of such
obstructions. In particular, it does not settle the sufficiency of
the anti-cupping condition or the two-quantifier theory of the
$\Delta^0_2$ degrees with join.

\appendix
\section{Relation with the no-least-join construction}\label{app:comparison}

The following correspondence identifies the structural ingredients
borrowed from~\cite[Section~6, pp.~2999--3015]{DGLM}.
The references in the right-hand column use the numbering of the
published version of~\cite{DGLM}.
The present proof is given in full, so none of these correspondences
is used in place of a proof of a transfer claim.

\bigskip
\begin{center}
\small
Structural correspondence with the DGLM no-least-join construction
\par\smallskip
\renewcommand{\arraystretch}{1.08}
\begin{tabular}{@{}p{.47\textwidth}p{.46\textwidth}@{}}
\hline
Present argument & Corresponding ingredient in \DGLM\\
\hline
Lemma~\ref{lem:modulus} and Section~\ref{subsec:budgets}
 & Sections 6.2.1 and 6.2.5; Step 3 of Section 6.3\\[4pt]
Lemma~\ref{lem:ancestry}
 & Results 6.10--6.18\\[4pt]
Lemma~\ref{lem:persistence} and Proposition~\ref{prop:Xi}
 & Lemmas 6.8--6.9, 6.19--6.20; Proposition 6.7\\[4pt]
Lemmas~\ref{lem:certaxioms} and~\ref{lem:Gamma}
 & Lemmas 6.21--6.24\\[4pt]
Lemma~\ref{lem:history} and Corollary~\ref{cor:nodeeper}
 & Lemmas 6.27--6.28\\[4pt]
Lemma~\ref{lem:available} and Corollary~\ref{cor:preaction}
 & Lemma 6.25 and Corollary 6.26\\[4pt]
Lemma~\ref{lem:permanent} and the limit coding argument
 & Lemma 6.29, Claims 6.30--6.32, and Section 6.4.9\\[4pt]
Lemma~\ref{lem:column}
 & Column-avoidance modification of the certificate rule and final verification\\
\hline
\end{tabular}
\end{center}
\bigskip

Three formulation choices are worth recording. First, the minimum in
\eqref{eq:modulus} is the beginning of the current stable interval;
taking the greatest stage in that interval would give $s$ and would
not yield the required approximation. Second, the safety corollary
is explicitly derived from the invariant at the \emph{previous}
stage. Third, final diagonalization uses the long root whose length
was chosen with the certificate. The coding functionals retain the full target $D$ and keep the
first-component and output lengths equal throughout.

\section*{Acknowledgments}

The main result was discovered by ChatGPT 6 Pro (OpenAI), which was
also used in developing the proof and preparing the manuscript.
The author takes full responsibility for the correctness of the
mathematical statements and proofs, for the accuracy of the citations
and attributions, and for the final content of the manuscript.

\end{document}